\documentclass[a4paper,reqno,10pt]{amsart}

\title{A pLaTeX example}

\usepackage{hyperref}
\usepackage{a4wide}

\usepackage{comment}
\usepackage{amssymb}
\usepackage{amstext}
\usepackage{amsmath}
\usepackage{amscd}
\usepackage{amsthm}
\usepackage{amsfonts}

\usepackage{graphicx}
\usepackage{latexsym}
\usepackage{mathrsfs}
\usepackage{longtable}
\usepackage{url}

\usepackage{enumitem}
\setlist[enumerate,1]{label={\upshape(\arabic*)}}
\setlist[enumerate,2]{label={\upshape(\alph*)}}

\numberwithin{table}{section}

\usepackage{tikz}
\usetikzlibrary{cd,arrows,matrix,backgrounds,shapes,positioning,calc,decorations.markings,decorations.pathmorphing,decorations.pathreplacing}
\tikzset{blackv/.style={circle,fill=black,inner sep=3pt,outer sep=3pt},
         whitev/.style={circle,fill=white,draw=black,inner sep=3pt,outer sep=3pt},
         blabel/.style={circle,draw=black,inner sep=1.5pt,outer sep=0pt},
         redv/.style={circle,fill=red,inner sep=3pt,outer sep=3pt},
         bluev/.style={circle,fill=blue,inner sep=3pt,outer sep=3pt},
         block/.style={draw,rectangle split,rectangle split horizontal,rectangle split parts=#1},
         symbol/.style={
           draw=none,
           every to/.append style={
             edge node={node [sloped, allow upside down, auto=false]{$#1$}}}}
}

\usepackage{array}
\newcolumntype{C}{>{$}c<{$}}
\newcolumntype{x}[1]{>{\centering\arraybackslash\hspace{0pt}}m{#1}}
\usepackage{makecell}

\newtheorem{theorem}{Theorem}[section]

\newtheorem{corollary}[theorem]{Corollary}
\newtheorem{lemma}[theorem]{Lemma}
\newtheorem*{lemma*}{Lemma}
\newtheorem*{theorem*}{Theorem}
\newtheorem{proposition}[theorem]{Proposition}
\newtheorem{definition-proposition}[theorem]{Definition-Proposition}

\theoremstyle{definition}
\newtheorem{definition}[theorem]{Definition}

\newtheorem{remark}[theorem]{Remark}
\newtheorem{example}[theorem]{Example}

\newtheorem*{ack}{Acknowledgments}

\renewcommand{\AA}{\mathcal{A}}

\newcommand{\CC}{\mathcal{C}}

\newcommand{\DD}{\mathcal{D}}

\newcommand{\HH}{\mathcal{H}}

\renewcommand{\SS}{\mathcal{S}}
\newcommand{\TT}{\mathcal{T}}
\newcommand{\TTT}{\mathsf{T}}
\newcommand{\UU}{\mathcal{U}}

\newcommand{\WW}{\mathcal{W}}

\newcommand{\XX}{\mathcal{X}}

\newcommand{\Z}{\mathbb{Z}}

\newcommand{\Ext}{\operatorname{Ext}\nolimits}

\newcommand{\Hom}{\operatorname{Hom}\nolimits}

\newcommand{\End}{\operatorname{End}\nolimits}

\newcommand{\RHom}{\mathbf{R}\strut\kern-.2em\operatorname{Hom}\nolimits}

\newcommand{\Kernel}{\operatorname{Ker}\nolimits}
\newcommand{\Cokernel}{\operatorname{Coker}\nolimits}

\newcommand{\ann}{\operatorname{ann}\nolimits}

\newcommand{\coker}{\Cokernel}

\renewcommand{\ker}{\Kernel}

\newcommand{\se}{\subseteq}

\DeclareMathOperator{\moduleCategory}{\mathsf{mod}} \renewcommand{\mod}{\moduleCategory}

\DeclareMathOperator{\ind}{\mathsf{ind}}
\DeclareMathOperator{\simp}{\mathsf{sim}}

\DeclareMathOperator{\ftors}{\mathsf{f-tors}}

\DeclareMathOperator{\sbrick}{\mathsf{sbrick}}
\DeclareMathOperator{\wide}{\mathsf{wide}}

\DeclareMathOperator{\tors}{\mathsf{tors}}

\DeclareMathOperator{\sttilt}{\mathsf{s}\tau\mathsf{-tilt}}

\DeclareMathOperator{\Indec}{\mathsf{Indec}}

\DeclareMathOperator{\Fac}{\mathsf{Fac}}

\DeclareMathOperator{\Filt}{\mathsf{Filt}}
\DeclareMathOperator{\thick}{\mathsf{thick}}

\DeclareMathOperator{\add}{\mathsf{add}}

\DeclareMathOperator{\id}{\mathsf{id}}

\DeclareMathOperator{\Ind}{\mathsf{Ind}}
\DeclareMathOperator{\Res}{\mathsf{Res}}

\newcommand{\iso}{\cong}

\numberwithin{equation}{section}

\begin{document}
\title[Clifford's theorem for bricks]{Clifford's theorem for bricks}

\author[Y. Kozakai]{Yuta Kozakai}
\address{Y. Kozakai: Department of Mathematics, Tokyo University of Science 1-3, Kagurazaka, Shinjuku-ku, Tokyo, 162-8601, Japan}
\email{kozakai@rs.tus.ac.jp}

\author[A. Sakai]{Arashi Sakai}
\address{A. Sakai: Graduate School of Mathematics, Nagoya University, Chikusa-ku, Nagoya, 464-8602, Japan}
\email{m20019b@math.nagoya-u.ac.jp}

\subjclass[2020]{20C20, 16G10, 18E10}
\keywords{Clifford's theorem, bricks, semibricks, simple-minded collections, wide subcategories}

\begin{abstract}
Let $G$ be a finite group, $N$ a normal subgroup of $G$, and $k$ a field of characteristic $p>0$.
In this paper, we formulate the brick version of Clifford's theorem under suitable assumptions and prove it by using the theory of wide subcategories.
As an application of our theorem, we consider the restrictions of semibricks and two-term simple-minded collections under the assumption that the index of the normal subgroup $N$ in $G$ is a $p$-power.
\end{abstract}

\maketitle
\tableofcontents

\section{Introduction}\label{sec:intro}

One of the most fundamental theories in the representation theory of finite groups is Clifford theory introduced in   \cite{C}. The theory describes the relation between representations of a finite group $G$ and those of a normal subgroup $N$ of $G$.
One of the most important and famous theorems in this theory is the so-called Clifford's theorem, which explains the relation of simple modules over $kG$ and those over $kN$, where $k$ denotes a field.

\begin{theorem}[Clifford's theorem]
Let $G$ be a finite group, $N$ a normal subgroup of $G$ and $k$ a field. For any simple $kG$-module $S$, its restriction $\Res S$ to $N$ is a semisimple $kN$-module. Moreover, if we decompose $\Res S=T_1^{a_1}\oplus\cdots\oplus T_n^{a_n}$, where $T_1, \ldots, T_n$ are pairwise non-isomorphic simple modules, then the following statements hold.
\begin{enumerate}
    \item $G$ permutates \{$T_i^{a_i}\mid i=1,\ldots,n\}$ transitively.
    \item $\dim_kT_1=\cdots=\dim_kT_n$ and $a_1=\cdots=a_n$.
\end{enumerate}
\end{theorem}

Since this theorem was released, this has been used in many studies on the representation theory of finite groups.
In particular, this theorem has been generalized in a variety of contexts.

One the other hand, there are generalized concepts of simple modules and semisimple modules.
For a finite dimensional algebra $\Lambda$ over a field, a $\Lambda$-module $M$ is called a {\it brick} if its endomorphism ring is a division algebra, and called a {\it semibrick} if it is a direct sum of bricks and if there are no nonzero homomorphisms between any two pairwise non-isomorphic indecomposable summands of it. 
These notions are generalizations of simple modules and semisimple modules respectively, nowadays they are characterized in various ways (\cite{As, DIJ, DIRRT, Rin}).
To investigate bricks and semibricks is expected to develop representation theory for a variety of reasons. In fact, semibricks correspond bijectively with wide subcategories of $\mod \Lambda$ (\cite{Rin}).
Moreover, then semibricks correspond bijectively with support $\tau$-tilting modules, two-term silting complexes, torsion classes, two-term simple minded-collections and more under the assumption that there exist only finitely many isomorphism classes of support $\tau$-tilting modules over $\Lambda$ (\cite{AIR, As, BY, KY}).
Also, there exist only finitely many isomorphism classes of support $\tau$-tilting modules over $\Lambda$ if and only if so do bricks over $\Lambda$ (\cite{DIJ}).
There are many other important reasons to consider bricks and semibricks, hence the studies on bricks and semibricks make sense and should be done.
That's why, given a brick version of Clifford's theorem, the representation theory would develop but the restrictions of bricks are not semibricks in general unfortunately (see Example \ref{example:a4ands4}). However, we find that the brick version of Clifford's theorem holds under the appropriate assumptions using the language of wide subcategories of $\mod kG$.
Here we say that a wide subcategory $\WW$ of $\mod kG$ is {\it stable under $k[G/N]\otimes_k-$} if we have that $k[G/N]\otimes_k W \in \WW$ for any object $W$ in $\WW$.

\begin{theorem}[see Theorem \ref{thm:main}]\label{main-thm1}
Let $G$ be a finite group, $N$ a normal subgroup of $G$ and $k$ a field of characteristic $p>0$. Then, for any brick $S$ appearing as a simple object in a wide subcategory $\WW$ of $\mod kG$ stable under $k[G/N]\otimes_k-$, the following hold.
\begin{enumerate}
    \item $\Res S$ is a semibrick, here $\Res$ means the restriction functor from $\mod kG$ to $\mod kN$.
    \item If we decompose $\Res S=T_1^{a_1}\oplus\cdots\oplus T_n^{a_n}$, where $T_1, \ldots, T_n$ are pairwise non-isomorphic indecomposable modules, then the following statements hold.
\begin{enumerate}
    \item $G$ permutates \{$T_i^{a_i}\mid i=1,\ldots,n\}$ transitively.
    \item $\dim_kT_1=\cdots=\dim_kT_n$ and $a_1=\cdots=a_n$.
\end{enumerate}
\end{enumerate}
\end{theorem}

We consider the case that the index of $N$ in $G$ is a power of $p$ (recall that $p>0$ denotes the characteristic of the field $k$).
Then the assumption on each wide subcategory of Theorem \ref{main-thm1} is satisfied automatically. 
Moreover, the restriction of semibricks to $kN$ is semibricks again.
That is, we have the following theorem.

\begin{theorem}[see Corollary \ref{p-group} and Proposition \ref{prop:sbrick}]
Let $G$ be a finite group, $N$ a normal subgroup of $G$ and $k$ a field of characteristic $p>0$. 
For any brick $S$ and semibrick $M$ for $kG$, if the index of $N$ in $G$ is a power of $p$, then the following hold:
\begin{enumerate}
    \item $\Res S$ and $\Res M$ are semibricks, here $\Res$ means the restriction functor from $\mod kG$ to $\mod kN$.
    \item If we decompose $\Res S=T_1^{a_1}\oplus\cdots\oplus T_n^{a_n}$, where $T_1, \ldots, T_n$ are pairwise non-isomorphic indecomposable modules, then the following statements hold.
\begin{enumerate}
    \item $G$ permutates \{$T_i^{a_i}\mid i=1,\ldots,n\}$ transitively.
    \item $\dim_kT_1=\cdots=\dim_kT_n$ and $a_1=\cdots=a_n$.
\end{enumerate}
\end{enumerate}

\end{theorem}

Brou\'{e}'s conjecture is still open problem. 
The conjecture predicts that the principal block $B_0(kG)$ of $G$ is derived equivalent to the one $B_0(kN_G(P))$ of $N_G(P)$ under the assumption that $P$ is an abelian group, where $P$ means a Sylow $p$-subgroup of $G$.
The conjecture does not hold in general without the assumption that $G$ has an abelian Sylow $p$-subgroup, but it is expected that there exist examples that Brou\'{e} conjecture holds even if $G$ has a non abelian Sylow $p$-subgroup.
Recently, it has been hoped that an analogue of the conjecture holds under the assumption that the finite group $G$ has a Sylow $p$-subgroup isomorphic to a non abelian metacyclic $p$-group $M_{n+1}(p)(\cong C_{p^n}\rtimes C_p)$ of order $p^{n+1}$, here $n \geq 2$. 
To be more specific, in this case, there exist a normal subgroup $N$ such that the index of $N$ in $G$ is $p$ and a Sylow $p$-subgroup $P$ of $N$ is isomorphic to a cyclic group of order $p^n$ (see \cite[Lemma 2.6]{HKK}), and we strongly believe that the principal block $B_0(kG)$ of $G$ is derived equivalent to the one $B_0(kN_G(P))$ of $N_G(P)$.
In fact, it is proved that there exists an isometry between $B_0(kG)$ and $B_0(kN_G(P))$ satisfying the separation condition preserving heights \cite[Theorem 1.1]{HKK}. This isometry is induced from that between $B_0(kN)$ and $B_0(kN_N(P))$ naturally.

Thus we hope that we can construct a derived equivalence between $B_0(kG)$ and $B_0(kN_G(P))$ from the one between $B_0(kN)$ and $B_0(kN_N(P))$ and there are some studies in which we construct objects for $B_0(kG)$ related to derived equivalence, such as support $\tau$-tilting modules and tilting complexes, from those for $B_0(kN)$ (\cite{KK21, KK23, Ko}).
In this paper, we consider simple-minded collections.
Simple-minded collections were first studied by Rickard in \cite{Ric2}, they are important for the considerations for derived equivalences for symmetric algebras.
Indeed, in many cases of the proof of Brou\'{e}'s conjecture, they use Okuyama method introduced in \cite{Oku}.
In this method, it is essential to find preimages of simple modules under the stable equivalences of Morita type induced by derived equivalences.
Also we know the way to find the preimages of the simple modules from simple-minded collections (\cite{Ric, Ric2}).
In particular, one case to consider is the case of two-term simple-minded collections because they are easier to calculate the images into the stable categories from derived categories.
From such viewpoints, giving the similarity between simple-minded collections for $kG$ and the one for $kN$ is key to the solution of an analogue of Brou\'{e}'s conjecture mentioned above.
Therefore we have the following result on simple-minded collections which can be applied to the above metacyclic situations as an application of Theorem \ref{main-thm1}.
Here, for a simple-minded collections $\XX$ in $D^b(\mod kG)$, we denote the set of indecomposable direct summands of $\Res X$ for all objects $X$ in $\XX$ by $\Res \XX$.

\begin{theorem}[see Theorem \ref{thm:main2}]
Let $k$ be a field of characteristic $p>0$, $G$ a finite group and $N$ a normal subgroup of $G$ of $p$-power index. Then, for two-term simple-minded collection $\mathcal{X}$ in $D^b(\mod kG)$, we have that $\Res \mathcal{X}$ is a two-term simple-minded collection in $D^b(\mod kN)$.

\end{theorem}

\medskip
\noindent
{\bf Organization.}
In Section \ref{sec:pre}, we collect basic definitions and facts. 
In Section \ref{sec:Clifford}, we give the brick version of Clifford's theorem. 
In Section \ref{sec:smc}, we consider the restriction of two-term simple minded collections for $kG$ to $kN$ under the assumption that the index of $N$ in $G$ is a power of $p$. 
In Appendix we collect some fact about subcategories for group algebras.

\medskip
\noindent
{\bf Conventions and notation.}
Throughout this paper, $k$ denotes a field of characteristic $p>0$. Let $G$ be a finite group and $N$ a normal subgroup of $G$. We denote by $\mod kG$ the category of finitely generated left $kG$-modules. We assume that all modules are finitely generated left modules and that all subcategories are full, additive, and closed under isomorphisms. 
We denote by $k_G$ the trivial $kG$-module, that is, the one-dimensional vector space on which each element $g\in G$ acts as the identity. We denote the induction functor and restriction functor by $\Ind$ and $\Res$ respectively: $\Ind:=\Ind_N^G:=kG\otimes_{kN} -: \mod kN \rightarrow \mod kG, \Res:=\Res^G_N: \mod kG\rightarrow \mod kN$ unless otherwise stated.

\section{Preliminaries}\label{sec:pre}

\subsection{Wide subcategories and semibricks}\label{sec:preriminary}

In this subsection, we collect basic definitions and properties of wide subcategories and semibricks. Let $\Lambda$ be a finite dimensional $k$-algebra.

\begin{definition}\label{def:basicdef}
  Let $\WW$ be a subcategory of $\mod\Lambda$.
  \begin{enumerate}
    \item $\WW$ is \emph{closed under extensions} if, for any short exact sequence 
    \[
    \begin{tikzcd}
      0 \rar & M_1 \rar & M \rar & M_2 \rar & 0
    \end{tikzcd}
    \]
    in $\mod\Lambda$, we have that $M_1, M_2 \in \WW$ implies $M \in \WW$.
    \item $\WW$ is \emph{closed under kernels (resp. cokernels)} if, for every morphism $\varphi \colon M_1 \to M_2$ with $M_1, M_2 \in \WW$, we have $\ker\varphi \in \WW$ (resp. $\coker\varphi\in\WW$).
    \item $\WW$ is a \emph{wide subcategory of $\mod\Lambda$} if $\WW$ is closed under extensions, kernels and cokernels.
  \end{enumerate}
\end{definition}

Every wide subcategory $\WW$ of $\mod\Lambda$ becomes an abelian category, and we always regard $\WW$ as such. Moreover a subcategory $\WW$ of $\mod\Lambda$ is a wide subcategory if and only if it satisfies the following conditions:
\begin{enumerate}
    \item $\WW$ is an abelian category,
    \item the inclusion $\WW\to\mod\Lambda$ is an exact functor,
    \item $\WW$ is closed under extensions in $\mod\Lambda$.
\end{enumerate}

In an abelian category $\AA$, an object $S$ is called a \emph{simple object} if every monomorphism $X\to S$ in $\AA$ is zero or an isomorphism, equivalently, every epimorphism $S\to X$ in $\AA$ is zero or an isomorphism. We denote by $\simp\AA$ the set of isoclasses of simple objects in $\AA$. 

\begin{definition}
\begin{enumerate}
    \item A $\Lambda$-module $S$ is a \emph{brick} if $\End_\Lambda(S)$ is a division algebra, that is, any nonzero endomorphism of $S$ is an ismorphism. 
    \item A set $\SS$ of isoclasses of bricks in $\mod\Lambda$ is a \emph{semibrick} in $\mod\Lambda$ if it satisfies the following: for any $S\neq S'\in\SS$, we have $\Hom_\Lambda(S,S')=0$. 
    \item A $\Lambda$-module $M$ is a \emph{semibrick} if there is a semibrick $\{S_1,\ldots,S_n\}$ in $\mod\Lambda$ such that $M\iso S_1^{a_1}\oplus\cdots\oplus S_n^{a_n}$ for some $a_i\in\Z_{>0}$. 
\end{enumerate}
\end{definition}

For a collection $\SS$ of objects in $\mod\Lambda$, we say that a $\Lambda$-module $X$ \emph{has an $\SS$-filtration} if there is a chain of submodules
\[
 0=X_0\se X_1 \se\cdots\se X_{n-1}\se X_n=X
\]
such that for any $i\in\{1,\ldots,n\}$, we have $X_i/X_{i-1}\iso S$ for some $S\in\SS$. We denote by $\Filt\SS$ the subcategory of $\mod\Lambda$ consisting of $X$ which has a $\SS$-filtration. Then $\Filt\SS$ is the smallest extension-closed subcategory of $\mod\Lambda$ containing $\SS$. 

The following result states that a semibrick is exactly the set of isoclasses of all simple objects in a wide subcategory.

\begin{theorem}[{\cite[1.2]{Rin}, \cite[Theorem 2.5]{Eno}} for exact categories]\label{thm:rin}
There are bijective correspondences between the following two sets:
\begin{enumerate}
    \item the set of semibricks in $\mod\Lambda$,
    \item the set of wide subcategories of $\mod\Lambda$.
\end{enumerate}
The map from $(1)$ to $(2)$ is given by $\SS\mapsto\Filt\SS$. The converse is given by $\WW\mapsto\simp\WW$.
\end{theorem}

To give a proof of Theorem \ref{thm:main}, we use the following well-known result.

\begin{proposition}\label{prop:simple}
Let $\AA$ be an abelian category and $S_1,\ldots, S_n$ simple objects in $\AA$. Consider an object $X=S_1^{a_1}\oplus\cdots\oplus S_n^{a_n}$ where $a_i$ denotes a positive integer. Then the following statements hold.
\begin{enumerate}
    \item Every epimorphism from $X$ splits.
    \item The subobject $S_i^{a_i}$ of $X$ is the unique largest subobject of $X$ isomorphic to a finite direct sum of copies of $S_i$.
\end{enumerate}
\end{proposition}

We use the following lemma later.

\begin{lemma}\label{lem:lemma-for-smc}
For $\Lambda$-modules $X$ and $Y$ and any $i\geq 0$, if $\Ext^i_{\Lambda}(X, Y)=0$, then we have that $\Ext^i_{\Lambda}(X, \Filt Y)=0$.
\end{lemma}
\begin{proof}
Let $M$ be an object in $\Filt Y$. We demonstrate the statement by induction of lengths of the filtration of $M$. If $M \iso Y$, there is nothing to prove. Assume there exists an exact sequence $0\rightarrow Y \rightarrow M \rightarrow Y \rightarrow 0$. Then, by applying the functor $\Ext^i_{\Lambda}(X, -)$, we have an exact sequence $\Ext^i_{\Lambda}(X, Y) \rightarrow \Ext^i_{\Lambda}(X, M)\rightarrow \Ext^i_{\Lambda}(X, Y)$, but this forces $\Ext^i_{\Lambda}(X, M)=0$ since $\Ext^i_{\Lambda}(X, Y)=0$.
Also for an exact sequence $0\rightarrow M_1 \rightarrow M \rightarrow Y \rightarrow 0$, by the similar argument as above, we have an exact sequence $\Ext^i_{\Lambda}(X, M_1) \rightarrow \Ext^i_{\Lambda}(X, M)\rightarrow \Ext^i_{\Lambda}(X, Y)$. Hence we can proceed the induction.
\end{proof}

\subsection{Induction functors and restriction functors}

In this subsection, we recall some results on the induction and restriction functors. 

\begin{proposition}[{for example, see \cite[Lemma 8.5, Lemma 8.6]{Al} and \cite[COROLLARY 3.3.2]{B}}]\label{prop:Al}
Let $G$ be a finite group and $H$ a subgroup of $G$.
For any $kG$-module $U$, $kH$-module $V$, the induction functor $\Ind:=\Ind_H^G$ and the restriction functor $\Res:=\Res^G_H$, the following hold.
\begin{enumerate}
\item
$\Ind(V\otimes_k \Res U)\cong (\Ind V)\otimes_k U$,
\item
$\Hom_{kH}(V, \Res U)\cong \Hom_{kG}(\Ind V, U)$,
\item
$\Hom_{kH}(\Res U, V)\cong \Hom_{kG}(U, \Ind V)$,
\item
$\Ext^{1}_{kH}(V, \Res U)\cong \Ext^{1}_{kG}(\Ind V, U)$,
\item
$\Ext^{1}_{kH}(\Res U, V)\cong \Ext^{1}_{kG}(U, \Ind V)$.
\end{enumerate}
\end{proposition}

From now on, for a normal subgroup $N$ of a finite group $G$, we abbreviate $\Ind_N^G$ with $\Ind$ and $\Res_N^G$ with $\Res$ unless stated otherwise. In case that a subgroup $N$ of $G$ is a normal subgroup, the following holds.

\begin{proposition}\label{prop:indres}
For any $kG$-module $M$, it holds that $\Ind\Res M \cong k[G/N]\otimes_k M$.
\end{proposition}

\begin{proof}
For a trivial $kN$-module $k_N$, we have that
$$ \Ind\Res M \cong \Ind(k_N \otimes_k \Res M)\cong (\Ind k_N) \otimes_k M \cong k[G/N]\otimes_k M. $$
\end{proof}

The following is so-called Mackey's decomposition formula for normal subgroups.
We recall that for a $kN$-module $U$ and an element $g$ in $G$, the set $gU:=\{gu\:|\: u\in U\}$ has a $kN$-module structure by setting $x(gu) := g(g^{-1}xgu)$ for $gu \in gU$ and $x \in N$.

\begin{proposition}[{\cite[Lemma 8.7]{Al}}]\label{prop:Mackey}
For any $kN$-module $U$, it holds that $\Res\Ind U \cong \bigoplus_{g}gU$, where $g$ runs over representatives of the coset $G/N$ in $G$.
\end{proposition}

\section{Clifford's theorem for bricks}\label{sec:Clifford}

In this section, we aim at giving the brick version of Clifford's theorem.
First, we define a $G$-invariance of a subcategory of $\mod kN$.
\begin{definition}\label{def:G-stable}
A subcategory $\CC$ of $\mod kN$ is \emph{$G$-invariant} if for any $g\in G$ and any $C\in \CC$, we have $gC\in\CC$. 
\end{definition}

\begin{remark}\label{rem:auto}
For a $G$-invariant subcategory $\CC$ of $\mod kN$, any element $g$ in $G$ induces an auto-equivalence $g\cdot-\colon \CC\to \CC$. 
\end{remark}

\begin{lemma}\label{prop:G-inv-of-indinverse}
Let $\CC$ be a subcategory of $\mod kG$. Then 
\[
\Ind^{-1}(\CC)=\{X\in\mod kN\mid \Ind X\in\CC\}
\]
is a $G$-invariant subcategory of $\mod kN$.
\end{lemma}
\begin{proof}
Let $X$ be an object in $\Ind^{-1}(\CC)$ and $g$ an element in $G$. Then we have $\Ind(gX)\iso\Ind X\in\CC$, hence $gX\in\Ind^{-1}(\CC)$. 
\end{proof}

\begin{proposition}\label{prop:ind-inverse}
Let $\WW$ be a wide subcategory of $\mod kG$. Then the subcategory $\Ind^{-1}(\WW)$ is a $G$-invariant wide subcategory of $\mod kN$. 
\end{proposition}
\begin{proof}
The $G$-invariance of $\Ind^{-1}(\WW)$ follows from Lemma \ref{prop:G-inv-of-indinverse}.
Since $\Ind$ is an exact functor, it is easy to show that $\Ind^{-1}(\WW)$ is a wide subcategory of $\mod kN$. 
\end{proof}

For a functor $k[G/N]\otimes_k-: \mod kG \rightarrow \mod kG$, we define a stability under the functor $k[G/N]\otimes_k -$ of a subcategory of $\mod kG$.
\begin{definition}\label{def:tensor-stable}
A subcategory $\DD$ of $\mod kG$ is \emph{stable under $k[G/N]\otimes_k-$} if for any $D\in\DD$, we have $k[G/N]\otimes_kD\in\DD$. 
\end{definition}

\begin{proposition}\label{prop:Res}
Let $\WW$ be a wide subcategory of $\mod kG$ stable under $k[G/N]\otimes_k-$. Then the restriction functor $\Res$ induces an exact functor $\WW\to\Ind^{-1}(\WW)$ between wide subcategories.
\end{proposition}

\begin{proof}
By Proposition \ref{prop:ind-inverse}, we have that $\Ind^{-1}(\WW)$ is a wide subcategory. The remaining follows easily from the isomorphism $\Ind\Res X\iso k[G/N]\otimes_kX$ for any $X\in\mod kG$. 
\end{proof}

\begin{proposition}\label{prop:G-inv-of-resinverse}
Let $\DD$ be a subcategory of $\mod kN$. Then 
\[
\Res^{-1}(\DD)=\{Y\in\mod kG\mid \Res Y\in\DD\}
\]
is stable under $k[G/N]\otimes_k-$.
\end{proposition}
\begin{proof}
Let $Y$ be an object in $\Res^{-1}(\DD)$. Then $\Res(k[G/N]\otimes_k Y)=\Res k[G/N]\otimes_k \Res Y$. Now $\Res k[G/N]$ is a direct sum of trivial $kN$-modules, which means that $\Res (k[G/N]\otimes_k Y)$ is a direct sum of copies of $\Res Y$. Thus $k[G/N]\otimes_k Y$ belongs to $\Res^{-1}(\DD)$.
\end{proof}

Here is our main theorem stated in Section \ref{sec:intro}.
\begin{theorem}\label{thm:main}
Let $\WW$ be a wide subcategory of $\mod kG$ stable under $k[G/N]\otimes_k-$. Then Clifford's theorem holds for any simple object $S$ in $\WW$, that is, the following hold.
\begin{enumerate}
    \item $\Res S$ is a semisimple object in $\Ind^{-1}(\WW)$, in particular, a semibrick.
    \item If we decompose $\Res S=T_1^{a_1}\oplus\cdots\oplus T_n^{a_n}$, where $T_1, \ldots, T_n$ are pairwise non-isomorphic indecomposable $kN$-modules, then the following statements hold.
\begin{enumerate}
    \item $G$ permutates \{$T_i^{a_i}\mid i=1,\ldots,n\}$ transitively.
    \item $\dim_kT_1=\cdots=\dim_kT_n$ and $a_1=\cdots=a_n$.
\end{enumerate}
\end{enumerate}
\end{theorem}

To give a proof of the above, we make one consideration.
In the setting of Theorem \ref{thm:main}, the restricted module $\Res S$ is an object in $\Ind^{-1}(\WW)$ by Proposition \ref{prop:Res}, hence we can take a simple subobject $T$ of $\Res S$ in  $\Ind^{-1}(\WW)$. We show the following lemma on this simple subobject $T$.

\begin{lemma}\label{lem:ss}
In the setting of Theorem \ref{thm:main}, for a simple object $S$ in $\WW$ and a simple subobject $T$ of $\Res S$ in $\Ind^{-1}(\WW)$, it holds that $\Res S$ is a direct summand of $\bigoplus_{g\in[G/N]}gT$.  
\end{lemma}

\begin{proof}
We consider the $kG$-homomorphism $\varphi\colon\Ind T=kG\otimes_{kN}T\to S$ which sends $g\otimes t$ to $gt$. (This is the morphism corresponding to the inclusion $T\hookrightarrow\Res S$ via the adjunction $\Ind\dashv\Res$.) Then $\Ind T$ belongs to $\WW$ because $T\in\Ind^{-1}(\WW)$. Since $S$ is simple in $\WW$, the nonzero morphism $\varphi$ is epimorphic. Then we have the epimorphism $\Res\varphi\colon\Res\Ind T\to \Res S$. Also, by Proposition \ref{prop:Mackey} we have that $\Res\Ind T\iso \bigoplus_{g\in[G/N]}gT$.
Since $\Ind^{-1}(\WW)$ is a $G$-invariant wide subcategory by Proposition \ref{prop:ind-inverse},
we have that $\Res\Ind T\iso \bigoplus_{g\in[G/N]}gT$ is a semisimple object in $\Ind^{-1}(\WW)$ by Remark \ref{rem:auto}. By Proposition \ref{prop:simple} (1), the epimorhism $\Res\varphi\colon\Res\Ind T\to \Res S$ splits.
Therefore we obtain the desired result. 

\end{proof}

\begin{proof}[Proof of Theorem \ref{thm:main}]
(1) follows from Lemma \ref{lem:ss}. We show (2). For any $g\in G$ and any $i$, the module $gT_i$ is an indecomposable direct summand of $\Res S$ (see \cite[Lemma 2.5]{KK23-2}). Then there is an integer $1\leq i_g\leq n$ such that $gT_i\iso T_{i_g}$. Note that $g$ permutates $\{T_1,\ldots,T_n\}/\iso$. 

We show that if $gT_i\iso T_j$ holds, then we have $gT_i^{a_i}=T_j^{a_j}$. Since $g$ induces an autoequivalence of $\Ind^{-1}(\WW)$ by Remark \ref{rem:auto} and Lemma \ref{prop:G-inv-of-indinverse}, we have the demposition $\Res S=gT_1^{a_1}\oplus\cdots\oplus gT_n^{a_n}$ into simple objects in $\ind^{-1}(\WW)$. By Proposition \ref{prop:simple} (2), we have that $gT_i^{a_i}$ is the unique largest subobject of $\Res S$ isomorphic to a direct sum of copies of $gT_i\iso T_j$. Thus we have $gT_i^{a_i}=T_j^{a_j}$.

 For any $j$, there is $h\in G$ such that $T_j\iso h T_1$ by applying Lemma \ref{lem:ss} to the simple subobject $T_1$ of $\Res S$. By the previous paragraph, we have $hT_1^{a_1}=T_j^{a_j}$, hence (a) holds. By $T_j\iso h T_1$ and $hT_1^{a_1}=T_j^{a_j}$, we have $\dim_kT_j=\dim_kT_1$ and $a_1=a_j$, that is, (b) holds. 
\end{proof}

Note that Theorem \ref{thm:main} recovers the ordinary Clifford's theorem by taking $\WW$ as $\mod kG$ suitably. We consider the case that the index of $N$ in $G$ is a power of $p$.
\begin{lemma}\label{lem:p-group}
If $G/N$ is a $p$-group, then every wide subcategory of $\mod kG$ is stable under $k[G/N]\otimes_k-$. 
\end{lemma}

\begin{proof}
Let $\WW$ be a wide subcategory of $\mod kG$ and $X$ an object in $\WW$. Since $G/N$ is a $p$-group, $k[G/N]$ has a filtration by trivial modules $k_{G/N}$. Then $k[G/N]\otimes_kX$ has a filtration by copies of $X$ since $-\otimes_kX$ is an exact functor and $k_{G/N}\otimes_k X\iso X$. Thus $k[G/N]\otimes_k X$ belongs to $\WW$ since $\WW$ is closed under extensions. 
\end{proof}

\begin{corollary}\label{p-group}
If $G/N$ is a $p$-group, then Clifford's theorem holds for any brick in $\mod kG$. 
\end{corollary}

\begin{proof}
Let $S$ be a brick in $\mod kG$. Then $S$ is a simple object in a wide subcategory $\Filt S$ of $\mod kG$ by Theorem \ref{thm:rin}. Then $\Filt S$ is stable under $k[G/N]\otimes_k-$ by Lemma \ref{lem:p-group}. Thus we can apply Theorem \ref{thm:main} to $S$. 
\end{proof}

\begin{remark}
Let $R$ be a commutative ring. As in the case of a simple module, it can be shown that for any brick $S$ in $\mod RG$, its annihilator ideal $\ann_R(S)$ is a maximal ideal of $R$ (see \cite[Chapter 1, Exercise 15]{W}). 
Thus it follows from Corollary \ref{p-group} that Clifford's theorem holds for $S$ if the field $R/\ann_R(S)$ has characteristic $p>0$ and the index of $N$ is a power of $p$. In particular, Clifford's theorem holds for any brick in $\mod RG$ when $R$ is a local ring whose residue field has characteristic $p>0$ and the index of $N$ is a power of $p$. 
\end{remark}

\begin{example}\label{example:a4ands4}
Let $G=\mathfrak{S}_4, N=A_4$ and $k$ an algebraically closed field of characteristic $2$.
Then $kG$ has two simple modules $k_{G}$ and $S_{2}$, and $kN$ has three simple modules $k_{N}, T_1$ and $T_2$.
Here we remark that the dimension of $S_2$ is two and those of the others are all one,
and that $gT_1\cong T_2$, where we denote a complete set of representatives of $G/N$ by $\{e, g\}$.
All bricks for $kG$ are as follows (see \cite[Example 3.11]{KK23-2}):
$$
k_{G},
S_2,
\begin{bmatrix} S_2\\ k_{G} \\ k_{G}\end{bmatrix},
\begin{bmatrix} k_{G}\\ S_2\end{bmatrix},
\begin{bmatrix} S_2 \\ k_{G}\end{bmatrix},
\begin{bmatrix} k_{G}\\ k_{G} \\ S_2\end{bmatrix}.$$
Then the respective restricted modules to $kN$ are as follows:
$$\Res k_G=k_N,
\Res S_2=T_1\oplus T_2,
\Res \begin{bmatrix} S_2\\ k_{G} \\ k_{G}\end{bmatrix}
=  \begin{bmatrix} T_1\\ k_N\end{bmatrix} \oplus \begin{bmatrix} T_2 \\ k_N\end{bmatrix}
= \begin{bmatrix} T_1\\ k_N\end{bmatrix} \oplus g\begin{bmatrix} T_1 \\ k_N\end{bmatrix},$$
$$
\Res\begin{bmatrix} k_{G}\\ S_2\end{bmatrix}
=
\begin{bmatrix} k_N \\ T_1\oplus T_2\end{bmatrix},
\Res\begin{bmatrix} S_2 \\k_{G}\end{bmatrix}
=
\begin{bmatrix} T_1\oplus T_2 \\ k_N\end{bmatrix},
\Res \begin{bmatrix} k_{G} \\ k_{G}\\ S_2 \end{bmatrix}
=  \begin{bmatrix} k_N\\ T_1\end{bmatrix} \oplus \begin{bmatrix} k_N \\ T_2\end{bmatrix}
= \begin{bmatrix} k_N\\ T_1\end{bmatrix} \oplus g\begin{bmatrix} k_N \\ T_1\end{bmatrix}.$$
\end{example}

\begin{remark}
The hypothesis that $G/N$ is a $p$-group in Corollary \ref{p-group} is essential.
For example, consider the case $G=\mathfrak{S}_4, N=A_4, N_1=\{(1), (12)(34), (13)(24), (14)(23) \}$ and $p=2$.
(In this case, the quotient group $N/N_1$ is a $p'$-group and $G/N_1$ is not neither a $p$-group nor $p'$-group.)
Then $\begin{bmatrix} k_{N} \\ T_2\end{bmatrix}$ is a brick for $kN$, but $\Res^N_{N_1} \begin{bmatrix} k_{N} \\ T_2\end{bmatrix}\cong\begin{bmatrix} k_{N_1} \\ k_{N_1}\end{bmatrix}$ is not a semibrick.
Moreover $\begin{bmatrix} k_{G} \\ S_2\end{bmatrix}$ is a brick for $kG$, but $\Res^G_{N_1} \begin{bmatrix} k_{G} \\ S_2\end{bmatrix}\cong\begin{bmatrix} k_{N_1} \\ k_{N_1}\oplus k_{N_1}\end{bmatrix}$ is not a semibrick.

\end{remark}

\begin{remark}\label{rem:other-example}
We will see another example of wide subcategories which Theorem \ref{thm:main} can be applied in Theorem \ref{thm:compati}.
\end{remark}

The following proposition gives a certain inverse of Theorem \ref{thm:main}.
\begin{proposition}\label{prop:converse}
Let $G$ be a finite group, $N$ a normal subgroup of $G$, and $V$ be an indecomposable $kG$-module.
We consider the following two conditions.
\begin{enumerate}
\item
There exists a wide subcategory of $\mod kG$ stable under $k[G/N]\otimes_k-$ having $V$ as a simple object.
\item
$\Res V$ is a semibrick.
\end{enumerate}
Then \textnormal{(1)} means \textnormal{(2)}. Moreover if $p$ does not divide the index of $N$ in $G$, then \textnormal{(2)} means \textnormal{(1)}.
\end{proposition}
\begin{proof}
By Theorem \ref{thm:main}, (1) means (2). 
We show that (2) means (1) under the assumption that $p$ does not divide the index of $N$ in $G$.
Suppose that $\Res V \iso S_1^{a_1}\oplus \cdots \oplus S_n^{a_n}$ where $\{S_1,\ldots,S_n\}$ is a semibrick in $\mod kN$.
By Theorem \ref{thm:rin}, we have that $\WW:=\Filt\{S_1, \ldots, S_n\}$ is a wide subcategory of $\mod kN$ and $S_1, \ldots, S_n$ are simple objects in $\WW$. This implies that $\Res V$ is semisimple in $\WW$.
By Proposition \ref{prop:G-inv-of-resinverse}, we have that $\Res^{-1}(\WW)$ is a wide subcategory of $\mod kG$ which is stable under $k[G/N]\otimes_k -$.
Hence it is enough to show that $V$ is semisimple in $\Res^{-1}(\WW)$, which means that $V$ is the simple object in $\Res^{-1}(\WW)$ because $V$ is the indecomposable $kG$-module.
We consider an arbitrary injection $\iota: W \hookrightarrow V$ in $\Res^{-1}(\WW)$. We show that $\iota$ is a section.
The morphism $\Res\iota: \Res W \hookrightarrow \Res V$ is a split monomorphism by the dual of Proposition \ref{prop:simple} (1), that is, there exists a morphism $\pi :\Res V \rightarrow \Res W$ such that $\pi \circ \iota = \id_W$.
Here we consider the map $\pi':=\frac{1}{(G:N)}\sum_{g\in [G/N]}g\pi g^{-1}: V\rightarrow W$.
We can easily check that the map $\pi'$ is the $kG$-homomorphism and that $\pi'\circ \iota=\id_W$.
Therefore we have that $\iota$ is the split monomorphism in $\Res^{-1}(\WW)$, which concludes that $V$ is semisimple.
\end{proof}

\begin{remark}
In Proposition \ref{prop:converse}, the assumption that $p$ dose not divide the index of $N$ in $G$ is essential.
In fact, without the assumption it can happen that there does not exist a wide subcategory of $\mod kG$ such that $V$ is a simple object in $\WW$ even if $\Res V$ is a semibrick;
In the setting of Example \ref{example:a4ands4}, the $kG$-module $\begin{bmatrix}k_G\\k_G\end{bmatrix}$ is not a simple object in any wide subcategory of $\mod kG$ since it is not a brick, but $\Res^{G}_N\begin{bmatrix}k_G\\k_G\end{bmatrix}=k_N\oplus k_N$ and $\Res^{G}_{N_1}\begin{bmatrix}k_G\\k_G\end{bmatrix}=\Res^{N}_{N_1}(\Res^{G}_{N}\begin{bmatrix}k_G\\k_G\end{bmatrix})=k_{N_1}\oplus k_{N_1}$ are semibricks.
\end{remark}

At the end of this section, for any semibrick $\SS$ in $\mod kG$, we show that the restriction $\Res\SS$ is a semibrick in $\mod kN$ if $G/N$ is a $p$-group and give some compatibility with respect to $\Res$ and $\Ind$.
We prepare the following lemma.

\begin{lemma}\label{lem:vanish}
Let $X,Y$ be objects in $\mod kG$. Assume that $G/N$ is a $p$-group. If $\Hom_{kG}(X,Y)=0$ holds, then we have $\Hom_{kN}(\Res X,\Res Y)=0$. 
\end{lemma}
\begin{proof}
By Proposition \ref{prop:Al} (3) and Proposition \ref{prop:indres}, we have the isomorphisms
\[
\Hom_{kN}(\Res X,\Res Y)\iso\Hom_{kG}(X,\Ind\Res Y)\iso\Hom_{kG}(X,k[G/N]\otimes_k Y).
\]
Since $G/N$ is a $p$-group, $k[G/N]$ belongs to $\Filt(k_G)$. Then $k[G/N]\otimes_k Y$ belongs to $\Filt Y$ since $-\otimes_kY$ is an exact functor. By $\Hom_{kG}(X,Y)=0$ and Lemma \ref{lem:lemma-for-smc}, we have $\Hom_{kG}(X,k[G/N]\otimes_k Y)=0$.
\end{proof}

For $M\in\mod kG$, we denote by $\Indec(M)$ the set of isoclasses of indecomposable direct summands of $M$. For a semibrick $\SS$ in $\mod kG$, by Lemma \ref{lem:vanish}, we can set
\[
\Res \SS=\bigsqcup_{S\in \SS}\Indec(\Res S).
\]

\begin{proposition}\label{prop:sbrick}
Let $\SS$ be a semibrick in $\mod kG$. If $G/N$ is a $p$-group, then $\Res\SS$ is a semibrick in $\mod kN$. 
\end{proposition}

\begin{proof}
By Corollary \ref{p-group}, we have that $\Indec(\Res S)$ is a semibrick in $\mod kN$ for any $S\in\SS$. Therefore it is enough to show $\Hom_{kN}(\Res S,\Res S')=0$ for any $S\neq S'$ in $\SS$. This follows from Lemma \ref{lem:vanish}.
\end{proof}

For $\Lambda\in \{kN, kG\}$, we denote the set of semibricks in $\mod\Lambda$ by $\sbrick \Lambda$ and that of wide subcategories of $\mod kG$ by $\wide \Lambda$. 
We have that the well-defined map $\Ind^{-1}: \wide kG \rightarrow \wide kN$ by Proposition \ref{prop:ind-inverse}.
Moreover, under the assumption that $G/N$ is a $p$-group, we get the well-defined map $\Res: \sbrick kG\rightarrow \sbrick kN$ by Proposition \ref{prop:sbrick}.
On these two maps, we have the following.
\begin{proposition}
If $G/N$ is a $p$-group, then we have the following commutative diagram.
  \begin{equation*}
    \begin{tikzcd}
      \sbrick kG \rar["\Res"] \dar["\Filt"] & [2em] \sbrick kN \dar["\Filt"] \\
      \wide kG  \rar["\Ind^{-1}"] & \wide kN
    \end{tikzcd}
  \end{equation*}
\end{proposition}

\begin{proof}
First, we remark that the vertical maps are well-defined by Theorem \ref{thm:rin}.
We show that $\Ind^{-1}(\Filt \SS)=\Filt (\Res \SS)$ for $\SS\in \sbrick kG$.
For the proof that $\Filt (\Res \SS)$ is included in $\Ind^{-1}(\Filt S)$, 
it is enough to show that $\Res\SS\se\Ind^{-1}(\Filt S)$ because $\Ind^{-1}(\Filt S)$ is closed under extensions by Proposition \ref{prop:ind-inverse}. For any $T\in \Res \SS$ there exists $S\in \SS$ such that $T$ is a direct summand of $\Res S$.
Then $\Ind T$ is a direct summand of $\Ind(\Res S)\iso k[G/N]\otimes_k S\in\Filt\SS$. Since $\Filt\SS$ is closed under direct summands, we have $\Ind T\in\Filt\SS$. 
For the inverse inclusion, take $X\in \Ind^{-1}(\Filt \SS)$ arbitrarily. By the definition, we have that $\Ind X\in \Filt\SS$.
Then $\Res\Ind X$ belongs to $\Res(\Filt\SS)\se\Filt(\Res\SS)$. By Proposition \ref{prop:Mackey}, we have that $X$ is a direct summand of $\Res(\Ind X)$. Thus we have $X\in\Filt(\Res\SS)$.
\end{proof}

\section{Restrictions of simple-minded collections}\label{sec:smc}
In this section, as an application of Theorem \ref{thm:main}, we consider restrictions of simple-minded collections for $kG$.
First, we recall the definition of the simple-minded collections in the derived category $D^b(\mod \Lambda)$ for a finite dimensional $k$-algebra $\Lambda$.

\begin{definition}
Let $\Lambda$ be a finite dimensional $k$-algebra.
A set $\XX$ of isoclasses of objects in $D^b(\mod \Lambda)$ is called a {\it simple-minded collection} in $D^b(\mod \Lambda)$ if the following conditions are satisfied:
\begin{enumerate}
\item for any $X\in \XX$, the endomorphism ring $\End_{D^b(\mod \Lambda)}(X)$ is a division $k$-algebra,
\item for any $X_1\neq X_2$ in $\XX$, it holds that $\Hom_{D^b(\mod \Lambda)}(X_1, X_2)=0$,
\item for any $X_1$ and $X_2$ in $\XX$ and any negative integer $n$, it holds that $\Hom_{D^b(\mod \Lambda)}(X_1, X_2[n])=0$,
\item
the smallest thick subcategory $\thick\XX$ of $D^b(\mod \Lambda)$ containing $\XX$ coincides with $D^b(\mod \Lambda)$.
\end{enumerate}
\end{definition}

For a simple-minded collection $\XX$ in $D^b(\mod \Lambda)$, we say that $\XX$ is {\it two-term} if the $i$-th cohomology $H^i(X)$ is $0$ for any $i\neq -1, 0$ and $X\in \XX$.
On two-term simple minded collections, we know the following.

\begin{proposition}[{\cite[Corollary 5.5]{BY} and \cite[Theorem 3.3]{As}}]\label{prop:two-term-smd}
Let $\XX$ be a two-term simple minded collection in $D^b(\mod \Lambda)$.
Then every $X \in \XX$ belongs either $\mod \Lambda$ or $(\mod \Lambda)[1]$ up to isomorphisms in $D^b(\mod \Lambda)$.
Moreover $\XX \cap \mod \Lambda$ and $\XX[-1] \cap \mod \Lambda$ are semibricks.
\end{proposition}

For a two-term simple-minded collection $\mathcal{X}$ in $D^b(\mod kG)$, we set 
\[
\Res\XX=\bigcup_{X\in\XX}\Indec(\Res X).
\]

Here is our second main theorem.

\begin{theorem}\label{thm:main2}
Let $k$ be a field of characteristic $p>0$, $G$ a finite group and $N$ a normal subgroup of $G$ of index $p^n$ in $G$ for some positive integer $n$. Then, for any two-term simple-minded collection $\mathcal{X}$ in $D^b(\mod kG)$, we have that $\Res \mathcal{X}$ is a two-term simple-minded collection in $D^b(\mod kN)$.
\end{theorem}

To give a proof of the theorem, we use the following lemma.

\begin{lemma}\label{lem:vanish1}
Let $X,Y$ be objects in $\mod kG$. Assume that $G/N$ is a $p$-group. If $\Ext^1_{kG}(X,Y)=0$ holds, then we have $\Ext^1_{kN}(\Res X,\Res Y)=0$. 
\end{lemma}
\begin{proof}
By Proposition \ref{prop:Al}, this can be shown by the same manner as Lemma \ref{lem:vanish}.
\end{proof}

\begin{proof}[Proof of Theorem \ref{thm:main2}]
We can write $\XX$ as $\{ X_1, \ldots, X_l, Y_1[1], \ldots, Y_r[1] \}$ for semibricks $\{X_1,\ldots, X_l\}$ and $\{Y_1,\ldots, Y_r\}$ in $\mod kG$ by Proposition \ref{prop:two-term-smd}.

First, we need to show that $\End_{D^b(\mod kN)}(Z)$ is a division algebra for any $Z \in \Res \XX$ and that $\Hom_{D^b(\mod kN)}(Z, Z') = 0$ for any $Z\neq Z'$ in $\Res \XX$.
 For any indecomposable direct summands $Z$ and $Z'$ of $\Res X_i$ (or $\Res Y_j[1]$), we have that $$ \Hom_{kN}(Z, Z')= \begin{cases} \text{division algebra} & (Z\iso Z') \\ 0 &(Z \not\iso Z')\end{cases} $$ because $\Res X_i$ (or $\Res Y_j$) is a semibrick by Corollary \ref{p-group}.
Also for distinct $i$ and $i'$, by Lemma \ref{lem:vanish} we have $\Hom_{kN}(\Res X_i, \Res X_{i'})=0$, which means that $\Hom_{D^b(\mod kN)}(Z, Z')=0$ for any indecomposable summands $Z$ of $\Res X_i$ and $Z'$ of $\Res X_{i'}$. Similarly, $\Hom_{D^b(\mod kN)}(Z[1], Z'[1])=0$ holds for any indecomposable summands $Z[1]$ of $\Res Y_j[1]$ and $Z'[1]$ of $\Res Y_{j'}[1]$ for distinct $j$ and $j'$.
Also for $X_i$ and $Y_j$, it is clear that $\Hom_{D^b(\mod kN)}(\Res Y_j[1], \Res X_i)=0$. Moreover we have $$ \Hom_{D^b(\mod kN)}(\Res X_i, \Res Y_j[1]) \iso \Ext_{kN}^1(\Res X_i, \Res Y_j)=0, $$
by $\Ext_{kG}^1(X_i,Y_j)\iso\Hom_{D^b(\mod kG)}(X_i,Y_j[1])=0$ and Lemma \ref{lem:vanish1}. 
Therefore we have that $\End_{D^b(\mod kN)}(Z)$ is a division algebra for any $Z \in \Res \XX$ and $\Hom_{D^b(\mod kN)}(Z, Z') = 0$ for any pairwise non-isomorphic $Z, Z' \in \Res \XX$.

Next, we show that $\Hom_{D^b(\mod kN)}(Z, Z[n]) =0$ for any $Z, Z' \in \Res \XX$ and any negative integer $n$. It is enough to show that $\Hom_{D^b(\mod kN)}(\Res X_i, \Res Y_j) =0$ for any $1 \leq i \leq l$ and $1 \leq j \leq r$. Since $\XX$ is the simple-minded collection, it holds that $\Hom_{kG}(X_i, Y_j)=\Hom_{D^b(\mod kG)}(X_i, Y_j)=0$. 
Hence, by Lemma \ref{lem:vanish}, we have $\Hom_{D^b(\mod kN)}(\Res X_i, \Res Y_j)= \Hom_{kN}(\Res X_i, \Res Y_j)=0$.

Finally, we show that the thick subcategory $\thick (\Res \XX)$ coincides with $D^b(\mod kN)$. It is enough to show $\mod kN\subseteq\thick (\Res \XX)$. Let $M$ be an object in $\mod kN$. Since $\XX$ is the simple-minded collection in $D^b(\mod kG)$, we have $\mod kG\subseteq\thick \XX$. Therefore $\Ind M$ belongs to $\thick\XX$. Then we have $\Res\Ind M\in\Res(\thick\XX)\se\thick(\Res\XX)$. Thus we have that $M \in \thick (\Res \XX)$ because $\thick (\Res \XX)$ is closed under taking direct summands and $M$ is a direct summand of $\Res\Ind M$ by Proposition \ref{prop:Mackey}.
\end{proof}

\appendix
\section{Wide subcategories for normal subgroups}\label{sec:wide-subcategory}

In this section, we collect some results related to the operations $\Ind^{-1}$ and $\Res^{-1}$ considered in Section \ref{sec:Clifford}. 
At the end of this section, we give Theorem \ref{thm:compati} as an example of wide subcategories which can be applied to Theorem \ref{thm:main} as stated in Remark \ref{rem:other-example}

In the rest of this paper, we use the following notation. Let $\CC$ be a subcategory of $\mod kG$.
\begin{itemize}
    \item $\add \CC =\{X\in \mod kG\mid X \ \text{is a direct summand of some} \ C\in \CC \}$.
    \item $\Fac \CC =\{X\in \mod kG\mid \text{there exists an epimorphism}\ C \rightarrow X \ \text{for some}\ C\in \CC \}$.
    \item ${}^\perp\CC=\{X\in\mod kG\mid\Hom_{kG}(X,C)=0 \ \text{for any} \ C\in\CC\}$.
    \item $ \CC^{\perp}=\{X\in\mod kG\mid\Hom_{kG}(C,X)=0 \ \text{for any} \ C\in\CC\}$.
\end{itemize}

We mainly deal with torsion classes in this section.

\begin{definition}\label{def:tors}
  Let $\CC$ be a subcategory of $\mod\Lambda$.
  \begin{enumerate}
    \item $\CC$ is \emph{closed under quotients in $\mod\Lambda$} if, for every object $C \in \CC$, any quotient of $C$ in $\mod\Lambda$ belongs to $\CC$.
    \item $\CC$ is a \emph{torsion class in $\mod\Lambda$} if $\CC$ is closed under extensions and quotients in $\mod\Lambda$.
  \end{enumerate}
\end{definition}

First, we give some results on subcategories of $\mod kG$ and the ones of $\mod kN$.

\begin{proposition}\label{prop:Gchar}
Let $\CC$ be a subcategory of $\mod kN$ closed under direct summands. Then $\CC$ is $G$-invariant if and only if $\CC=\Ind^{-1}(\Res^{-1}(\CC))$ holds. 
\end{proposition}
\begin{proof}
By Proposition \ref{prop:G-inv-of-indinverse}, we only show the only if part. Suppose that $\CC$ is $G$-invariant. Let $X$ be an object in $\mod kN$. Then $\CC$ contains $X$ if and only if so does $\Res(\Ind X)$ by Proposition \ref{prop:Mackey}. Thus we obtain the desired result. 
\end{proof}

\begin{lemma}\label{lem:indres}
Let $X$ be an object in $\mod kG$. Then there is an exact sequence
\[
 (k[G/N]\otimes_k X)^m\to k[G/N]\otimes_k X \to X\to 0
\]
in $\mod kG$ for some $m\in\Z_{\geq 0}$.
\end{lemma}
\begin{proof}
Take an exact sequence
\[
 k[G/N]^m\to k[G/N]\to k_{G/N}\to 0
\]
in $\mod k[G/N]$. 
We can regard this exact sequence as the one in $\mod kG$.
Applying an exact functor $-\otimes_k X$ to the above, we obtain the desired result.
\end{proof}

The next is a result on the stability under the functor $k[G/N]\otimes_k -$ of subcategories.
(See Definition \ref{def:tensor-stable} for the definition of the stability under the functor $k[G/N]\otimes_k -$.) 

\begin{proposition}\label{prop:Gchar2}
Let $\CC$ be a subcategory in $\mod kG$ closed under cokernels. Then $\CC$ is stable under $k[G/N]\otimes_k-$ if and only if $\CC=\Res^{-1}(\Ind^{-1}(\CC))$ holds. 
In particular, if $\CC$ is a wide subcategory or torsion class, then the statement holds.
\end{proposition}
\begin{proof}

We remark that $\Ind(\Res X)\iso k[G/N]\otimes_k X$ for any $X\in \CC$ by Proposition \ref{prop:indres}.
The if part follows from Proposition \ref{prop:G-inv-of-resinverse}.
We show the only if part. Suppose that $\mathcal{C}$ is stable under $k[G/N]\otimes_k -$. Then it is easy to show $\CC\subseteq\Res^{-1}(\Ind^{-1}(\CC))$. Let $X$ be an object in $\Res^{-1}(\Ind^{-1}(\CC))$. Then $\Ind(\Res X)$ belongs to $\CC$. Since $\CC$ is closed under cokernels, we have $X\in\CC$ by Lemma \ref{lem:indres}. 

\end{proof}

We denote by $(\tors kN)^G$ the set of $G$-invariant torsion classes in $\mod kN$ and by $(\tors kG)^\star$ the set of torsion classes in $\mod kG$ stable under $k[G/N]\otimes_k-$.
Similarly, we denote by $(\wide kN)^G$ the set of $G$-invariant wide subcategories in $\mod kN$ and by $(\wide kG)^\star$ the set of wide subcategories in $\mod kG$ stable under $k[G/N]\otimes_k-$.
The following proposition explains the compatibility on above sets.

\begin{proposition}\label{prop:compati}
There exist isomorphisms
  \[
  \begin{tikzcd}[row sep = 0]
    (\tors kN)^G \rar[shift left, "\Res^{-1}"] & (\tors kG)^\star \lar[shift left, "\Ind^{-1}"] \\
  \end{tikzcd}
  \]
  and
    \[
  \begin{tikzcd}[row sep = 0]
    (\wide kN)^G \rar[shift left, "\Res^{-1}"] & (\wide kG)^\star \lar[shift left, "\Ind^{-1}"] \\
  \end{tikzcd}
  \]
as posets. 
\end{proposition}

\begin{proof}
By Lemma \ref{prop:G-inv-of-indinverse} and Proposition \ref{prop:G-inv-of-resinverse}, the maps $\Res^{-1}$ and $\Ind^{-1}$ are well-defined. Moreover these maps preserve poset structures clearly.  Also, Propositions \ref{prop:Gchar} and \ref{prop:Gchar2} show that these maps are bijections.
\end{proof}

The next is a result on the $G$-invariance of subcategories.
(See Definition \ref{def:G-stable} for the definition of the $G$-stability.)
For a subcategory $\CC$ of $\mod kN$, we denote by $\TTT(\CC)$ the smallest torsion class in $\mod kN$ containing $\CC$. For subcategories $\CC$ and $\DD$ of $\mod kN$, we denote by $\CC*\DD$ the subcategory of $\mod kN$ consisting of objects $X$ such that there is a short exact sequence
\[
0\to C\to X\to D\to 0
\]
in $\mod kN$ with $C\in\CC$ and $D\in\DD$.
For a torsion class $\TT$ in $\mod kN$, we set 
\[
\alpha\TT=\{X\in\TT\mid {}^\forall (f\colon T\to X) \ \text{in} \ \TT, \ker f\in \TT\}.
\]
In \cite{IT}, it is shown that $\alpha\TT$ is a wide subcategory of $\mod kN$.

\begin{proposition}\label{prop:Ginv}
Let $\CC$ and $\DD$ be $G$-invariant subcategories of $\mod kN$.
Then the following hold:
\begin{enumerate}
\item $\CC\cap\DD$ is $G$-invariant.
\item $\Fac\CC$ is $G$-invariant.
\item $\CC*\DD$ is $G$-invariant.
\item $\Filt\CC$ is $G$-invariant.
\item ${}^\perp\CC$ and $\CC^{\perp}$ are $G$-invariant.
\item $\TTT(\CC)$ is $G$-invariant.
\item If $\CC$ is a torsion class in $\mod kN$, then $\alpha\CC$ is $G$-stable.
\end{enumerate}
\end{proposition}

\begin{proof}
By Remark \ref{rem:auto}, an element $g\in G$ induces autoequivalences $g\cdot -$ of $\CC$ and $\DD$, hence it is easy to show (1), (2), (3) and (4). 

(5) We only show that ${}^\perp\CC$ is $G$-invariant because the other can be proved similarly. Let $X$ be an object in ${}^\perp\CC$. For any $g\in G$ and any $C\in\CC$, we have $\Hom_{kN}(gX,C)\iso\Hom_{kN}(X,g^{-1}C)=0$. Thus we have $gX\in {}^\perp\CC$.

(6) As is well known, it holds that $\TTT(\CC)={}^\perp(\CC^\perp)$. Hence the statement follows from (5).

(7) Let $X$ be an object in $\alpha\CC$. For any $g\in G$, we show $gX\in\alpha\CC$. For any morphism $f\colon C\to gX$ with $C\in\CC$, we only need to show $\ker f\in\CC$. Consider the exact sequence 
\[
0\to \ker f \to C \to gX
\]
in $\mod kN$. Since $g^{-1}\cdot-$ is an autoequivalence of $\mod kN$, it preserves kernels. Thus we have the exact sequence
\[
0\to g^{-1}(\ker f) \to g^{-1}C \to X.
\]
Since $\CC$ is $G$-invariant, $g^{-1}C$ is in $\CC$. Then $g^{-1}(\ker f)$ belongs to $\CC$ by $X\in\alpha\CC$. Since $\CC$ is $G$-invariant, we have $\ker f\iso g(g^{-1}(\ker f))\in\CC$ as desired.
\end{proof}

An object $X$ of $\mod kN$ is \emph{$G$-invariant} if $\add X$ is $G$-invariant. (see Definition \ref{def:G-stable}.) An object $Y$ in $\mod kG$ is \emph{stable under $k[G/N]\otimes_k-$} if $\add Y$ is stable under $k[G/N]\otimes_k-$. (see Definition \ref{def:tensor-stable}.)

For $kN$-modules $U$ and $V$, we write $U =_{\add} V$ if $\add U=\add V$. This relation is an equivalence relation. 
We denote the set of $\add$-equivalence classes of support $\tau$-tilting $\Lambda$-modules  by $\sttilt \Lambda$ for $\Lambda\in\{kN, kG\}$ (for the definition of support $\tau$-tilting modules, see \cite[DEFINITION 0.3]{AIR}).
Moreover we set
$$(\sttilt kN)^G:=\{M \in \sttilt kN\mid M\ \text{is} \ G\text{-invariant}\}$$
and
$$(\sttilt kG)^\star:=\{L \in \sttilt kG\mid S\otimes_k L \in \add L\ \text{for each simple}\ k[G/N]\text{-module}\ S \}.$$ 
Koshio and the first author showed the following result.
\begin{theorem}[{\cite[Theorem 1.3]{KK23-2}}]
The induction functor $\Ind_N^G: \mod kN \rightarrow \mod kG$ induces a poset isomorphism between $(\sttilt kN)^G$ and $(\sttilt kG)^\star$.
\end{theorem}

On the other hand, by \cite{AIR}, we know that there are one-to-one correspondence between $\sttilt \Lambda$ and $\ftors \Lambda$ for $\Lambda \in \{kN, kG\}$, where $\ftors \Lambda$ denotes the set of functorially finite torsion classes for $\Lambda$.
Now, we naturally wonder if there is some kind of correspondence between a subset of the set of torsion classes for $kN$ and that for $kG$ compatible with the above correspondence.
The following is a positive answer, where we put $(\ftors kN)^G:=\ftors kN \cap (\tors kN)^G$ and $(\ftors kG)^\star:= \ftors kG \cap (\tors kG)^\star$.

\begin{proposition}\label{prop:commute}
Let $X$ be a $G$-invariant support $\tau$-tilting $kN$-module. Then $\Fac(\Ind X)=\Res^{-1}(\Fac X)$ holds. In other words, there is the following commutative diagram:
  \begin{equation*}
    \begin{tikzcd}
      (\sttilt kN)^G \rar["\Ind"] \dar["\Fac"] & [2em] (\sttilt kG)^\star \dar["\Fac"] \\
      (\ftors kN)^G  \rar["\Res^{-1}"] & (\ftors kG)^\star 
    \end{tikzcd}
  \end{equation*}
\end{proposition}

\begin{proof}
We can easily check that the vertical maps in the above diagram are well-defined.
Let $Y$ be an object in $\mod kG$. Suppose that $Y$ belongs to $\Fac(\Ind X)$. Then there is a surjection $\oplus\Ind X\to Y$. Applying $\Res$, we have a surjection $\oplus\Res(\Ind X)\to \Res Y$. Since $X$ is $G$-invariant, $\Res(\Ind X)$ belongs to $\add X$. Thus $\Res Y$ belongs to $\Fac X$.

Suppose that $Y$ belongs to $\Res^{-1}(\Fac X)$, that is, $\Res Y\in\Fac X$. Then there is a surjection $\oplus X\to \Res Y$. Applying $\Ind$, we have a surjection $\oplus\Ind X\to \Ind\Res Y$. Then, by Lemma \ref{lem:indres} there is a surjection $\Ind\Res Y \to Y$. Thus there is a surjection $\oplus\Ind X\to Y$. 
\end{proof}

There are some studies on torsion classes for finite dimensional algebras.
One of our interests is a wide interval of torsion classes in $\mod kG$.
In the rest of this section, we consider wide intervals of torsion classes in $\mod kG$ and those in $\mod kN$.
First, we recall the definition of wide intervals.

\begin{definition}[{see \cite[Definition 4.1]{AP}}]
Let $\UU$ and $\TT$ be torsion classes in $\mod kG$ satisfying $\UU\se\TT$. We call
\[
[\UU,\TT]=\{\CC\in\tors kG\mid\UU\se\CC\se\TT\}
\]
an \emph{interval in $\tors kG$}. We call $\HH_{[\UU,\TT]}=\TT\cap\UU^{\perp}$ the \emph{heart of $[\UU,\TT]$}. We call an interval $[\UU,\TT]$ a \emph{wide interval} if $\HH_{[\UU,\TT]}$ is a wide subcategory of $\mod kG$. 
\end{definition}

One motivation for considering the wide intervals is due to the following result.

\begin{proposition}[{see \cite[Proposition 6.1 and Proposition 6.3]{AP} or \cite[Proposition 3.3]{ES}}]
Any wide subcategory of $\mod kG$ is realized as a heart of some interval in $\tors kG$.
\end{proposition}

\begin{proposition}\label{prop:2-3}
Let $[\UU,\TT]$ be an interval in $\tors kN$ and $\HH$ the heart of $[\UU,\TT]$. If two of $\UU, \TT$ and $\HH$ are $G$-invariant, then so is the third.  
\end{proposition}
\begin{proof}
This follows easily from Proposition \ref{prop:Ginv} and the equations $
\TT=\UU*\HH$ and $\UU=\TT\cap {}^{\perp}\HH$, see \cite[Lemma 2.8]{ES}.
\end{proof}

\begin{lemma}\label{lem:resper}
Let $\TT$ be a $G$-invariant torsion class in $\mod kN$. Then we have $\Res^{-1}(\TT^{\perp})=(\Res^{-1}(\TT))^{\perp}$.
\end{lemma}
\begin{proof}
Let $X$ be an object in $\Res^{-1}(\TT^{\perp})$. For any $Y\in\Res^{-1}(\TT)$, we show $\Hom_{kG}(Y,X)=0$. Let $f\colon Y\to X$ be a morphism in $\mod kG$. Then $\Res f=0$ holds by $\Res X\in\TT^{\perp}$ and $\Res Y\in\TT$. Therefore we have $f=0$.

Conversely, let $X$ be an object in $(\Res^{-1}(\TT))^{\perp}$. We need to show $\Res X\in\TT^{\perp}$. Let $Y$ be an object in $\TT$. By Proposition \ref{prop:Gchar}, we have $\Ind Y\in\Res^{-1}(\TT)$. Then $\Hom_{kN}(Y,\Res X)\iso\Hom_{kG}(\Ind Y,X)=0$ holds by $X\in(\Res^{-1}(\TT))^{\perp}$. Thus $\Res X$ belongs to $\TT^{\perp}$.
\end{proof}

\begin{proposition}\label{prop:inv}
Let $\TT$ and $\UU$ be $G$-invariant torsion classes in $\mod kN$ satisfying $\UU\se\TT$. Then the following hold:
\begin{enumerate}
    \item $\Res^{-1}(\HH_{[\UU,\TT]})=\HH_{[\Res^{-1}(\UU),\Res^{-1}(\TT)]}$.
    \item $\Ind^{-1}(\HH_{[\Res^{-1}(\UU),\Res^{-1}(\TT)]})=\HH_{[\UU,\TT]}$.
\end{enumerate}
In particular, $[\UU,\TT]$ is a wide interval if and only if so is $[\Res^{-1}(\UU),\Res^{-1}(\TT)]$. 
\end{proposition}
\begin{proof}
(1) By Lemma \ref{lem:resper}, we have
\[
\Res^{-1}(\TT)\cap(\Res^{-1}(\UU))^{\perp}=\Res^{-1}(\TT)\cap\Res^{-1}(\UU^{\perp})=\Res^{-1}(\TT\cap\UU^{\perp}).
\]
Thus we have $\Res^{-1}(\HH_{[\UU,\TT]})=\HH_{[\Res^{-1}(\UU),\Res^{-1}(\TT)]}$. 

(2) By Proposition \ref{prop:2-3}, the heart $\HH_{[\UU,\TT]}$ is $G$-invariant. Then the statement follows from (1) and Proposition \ref{prop:Gchar}. 

The latter part follows from that $\Res$ and $\Ind$ are exact functors. 
\end{proof}

\begin{theorem}\label{thm:compati}
Let $M_1$ and $M_2$ be $G$-invariant support $\tau$-tilting $kN$-modules, and $[\Fac M_1, \Fac M_2]$ a wide interval. Then $\Fac (\Ind M_1)$ and $\Fac (\Ind M_2)$ are functorially finite torsion classes, and $[\Fac (\Ind M_1), \Fac (\Ind M_2)]$ is a wide interval. Moreover the wide subcategory $\HH_{[\Fac (\Ind M_1), \Fac (\Ind M_2)]}$ is stable under $k[G/N] \otimes_k -$.
\end{theorem}

\begin{proof}
Let $M_1$ and $M_2$ be $G$-invariant support $\tau$-tilting $kN$-modules, and assume that $[\Fac M_1, \Fac M_2]$ is a wide interval.
By \cite[Theorem 1.3]{KK23-2}, each induced module $\Ind M_i$ is a support $\tau$-tilting $kG$-module and we get a $k[G/N]\otimes_k-$ stable torsion class $\Fac(\Ind M_i)$ by Proposition \ref{prop:commute}. 
Also, we have $\Fac(\Ind M_i) \cong \Res^{-1}(\Fac M_i)$ by Proposition \ref{prop:commute} again.
Therefore we have that $[\Fac(\Ind M_1), \Fac(\Ind M_2)]$ is a wide interval by Proposition \ref{prop:inv} and the assumption that $[\Fac M_1, \Fac M_2]$ is a wide interval.
The remaining statement follows from Proposition \ref{prop:G-inv-of-resinverse} and Lemma \ref{lem:resper}.
\end{proof}

\begin{ack}
The second author is supported by JSPS KAKENHI Grant Number JP22J20611.
\end{ack}


\begin{thebibliography}{DIRRT}

  \bibitem[AIR]{AIR}
  T. Adachi, O. Iyama, I. Reiten,
  \emph{$\tau$-tilting theory},
  Compos. Math. 150 (2014), no. 3, 415--452.
  
  \bibitem[Al]{Al}
  J. L. Alperin,
  \emph{Local representation theory},
  volume 11 of Cambridge Studies in Advanced Mathematics. Cambridge University Press, Cambridge, 1986.
  
  \bibitem[As]{As}
  S. Asai. \emph{Semibricks},
  Int. Math. Res. Not. IMRN, 2020(16):4993--5054, 2018.

  \bibitem[AP]{AP}
  S. Asai, C. Pfeifer,
  \emph{Wide subcategories and lattices of torsion classes},
  Algebr. Represent. Theory 25 (2022), no. 6, 1611--1629.
  
   \bibitem[B]{B}
   D. J. Benson,
   \emph{Representations and cohomology. I. Basic representation theory of finite groups and associative algebras}, Cambridge Studies in Advanced Mathematics, 30. Cambridge University Press, Cambridge, 1991.


  \bibitem[BY]{BY}
  T. Br\"{u}stle, D. Yang,
  \emph{Ordered exchange graphs},
  Advances in representation theory of algebras, 135--193, 
  EMS Ser. Congr. Rep., Eur. Math. Soc., Z\"{u}rich, 2013. 
  
  \bibitem[C]{C}
  A. H. Clifford, 
  \emph{Representations induced in an invariant subgroup},
  Ann. of Math. (2) 38 (1937), no. 3, 533--550.
  
  \bibitem[DIJ]{DIJ}
  L. Demonet, O. Iyama, G. Jasso,
  \emph{$\tau$-tilting finite algebras, bricks, and $g$-vectors},
  Int. Math. Res. Not. IMRN (2019), no. 3, 852--892.
  
  \bibitem[DIRRT]{DIRRT}
  L. Demonet, O. Iyama, N. Reading, I. Reiten, H. Thomas,
  \emph{Lattice theory of torsion classes: beyond $\tau$-tilting theory},
  Trans. Amer. Math. Soc. Ser. B10 (2023), 542--612.

  \bibitem[Eno]{Eno}
  H. Enomoto,
  \emph{Schur's lemma for exact categories implies abelian},
  J. Algebra 584 (2021), 260--269.

  \bibitem[ES]{ES}
  H. Enomoto, A. Sakai,
  \emph{ICE-closed subcategories and wide $\tau$-tilting modules},
  Math. Z. 300 (2022), no. 1, 541--577.
  
  \bibitem[HKK]{HKK}
  M. Holloway, S. Koshitani, N. Kunugi,
  \emph{Blocks with nonabelian defect groups which have cyclic subgroups of index $p$},
  Arch. Math. (Basel) 94 (2010), no. 2, 101--116. 

  \bibitem[IT]{IT}
   C. Ingalls, H. Thomas,
   \emph{Noncrossing partitions and representations of quivers},
   Compos. Math. 145 (2009), no.6, 1533--1562.
  
   \bibitem[KY]{KY}
  S. Koenig, D. Yang,
  \emph{Silting objects, simple-minded collections, $t$-structures and co-$t$-structures for finite-dimensional algebras}, Doc. Math. 19 (2014), 403--438.
  
   \bibitem[KK21]{KK21}
   R. Koshio, Y. Kozakai,
   \emph{On support $\tau$-tilting modules over blocks covering cyclic blocks},
   J. Algebra 580 (2021), 84--103.
   
   \bibitem[KK23]{KK23}
   R. Koshio, Y. Kozakai,
   \emph{Induced modules of support $\tau$-tilting modules and extending modules of semibricks over blocks of finite groups},
   J. Algebra 628 (2023), 524--544.
   
   \bibitem[KK23-2]{KK23-2}
   R. Koshio, Y. Kozakai,
   \emph{Normal subgroups and support $\tau$-tilting modules},
   arXiv:2301.04963.
   
   \bibitem[Ko]{Ko}
   Y. Kozakai,
   \emph{On tilting complexes over blocks covering cyclic blocks},
   Comm. Algebra 51 (2023), no. 6, 2435--2447.
  
  \bibitem[Lin]{Lin}
  M. Linckelmann,
  \emph{Stable equivalences of Morita type for self-injective algebras and $p$-groups}, 
  Math. Z. 223 (1996), no. 1, 87--100. 
  
  \bibitem[Oku]{Oku}
  T. Okuyama, 
  \emph{Some examples of derived equivalent blocks of finite groups},
  unpublished preprint (1997).
  
  \bibitem[Ric]{Ric}
  J. Rickard,
  \emph{Derived categories and stable equivalence},
  J. Pure Appl. Algebra 61 (1989), no. 3, 303--317.
  
  \bibitem[Ric2]{Ric2}
  J. Rickard,
  \emph{Equivalences of derived categories for symmetric algebras},
  J. Algebra 257 (2002), 460--481.
  
  \bibitem[Rin]{Rin}
  C. M. Ringel,
  \emph{Representations of $K$-species and bimodules},
  J. Algebra 41 (1976), no. 2, 269--302.
  
  \bibitem[W]{W}
  P. Webb, 
  \emph{A course in finite group representation theory}, Cambridge Studies in Advanced Mathematics, 161. Cambridge University Press, Cambridge, 2016.    
  
\end{thebibliography}
\end{document}